\documentclass[11pt]{amsart}
\usepackage{amsmath}
\usepackage{amsfonts}
\usepackage{amssymb}
\usepackage{bm}
\usepackage{mathrsfs}
\usepackage{bbding}
\usepackage{hyperref}
\usepackage{accents}
\usepackage{esint}
\usepackage{enumerate}
\usepackage[usenames]{color}

\newtheorem{teo}{Theorem}[section]
\newtheorem{lem}[teo]{Lemma}
\newtheorem{cor}[teo]{Corollary}
\newtheorem{prop}[teo]{Proposition}

\newtheorem{es}[teo]{Example}

\newtheorem{defi}[teo]{Definition}
\newtheorem{oss}[teo]{Remark}

\newcommand{\res}{\mathop{\hbox{\vrule height 7pt width .5pt depth 0pt
\vrule height .5pt width 6pt depth 0pt}}\nolimits}

\newcommand{\norm}[1]{\|{#1}\|}

\newcommand{\lip}{\mathrm{Lip}}
\newcommand{\BV}{\mathrm{BV}}
\renewcommand{\phi}{\varphi}
\newcommand{\R}{{\mathbb{R}}}

\newcommand{\bbR}{{\mathbb{R}}}
\newcommand{\bbN}{{\mathbb{N}}}
\def \R {{\mathbb {R}}}

\newcommand{\bbH}{\mathbb{H}}

\newcommand{\bbS}{\mathbb{S}}

\newcommand{\average}{{\mathchoice {\kern1ex\vcenter{\hrule height.4pt
width 6pt
depth0pt} \kern-9.7pt} {\kern1ex\vcenter{\hrule height.4pt width 4.3pt
depth0pt}
\kern-7pt} {} {} }}

\newcommand{\HH}{\mathcal{H}}
\newcommand{\LL}{\mathcal{L}}

\newcommand{\area}{\mathscr A}

\newcommand{\xvett}{\bm{X}}
\newcommand{\xva}{\xvett^\ast}

\DeclareMathOperator*{\argmin}{arg\,min}

\newcommand{\dive}{\mathrm{div}\,}

\newcommand{\ud}{\,\mathrm{d}}
\newcommand{\cc}{\mathscr M}

\newcommand{\atx}[1]{{#1}_{\tau,\xi}^\ast}

\newcommand{\ut}{\overline u_\tau^*}

\title[BV area minimizers]{BV minimizers of the area functional in the Heisenberg group under the bounded slope condition}
\author{Andrea Pinamonti}
\address{Andrea Pinamonti: Dipartimento di Matematica\\Universit\`a di Padova\\ Via Trieste 63\\ \\}
\email{pinamont@math.unipd.it}
\author{Francesco Serra Cassano}
\address{Francesco Serra Cassano: Dipartimento di Matematica\\Universit\`a di Trento\\ Via Sommarive 14\\ \\}
\email{cassano@science.unitn.it}

\author{Giulia Treu}
\address{Giulia Treu: Dipartimento di Matematica\\Universit\`a di Padova\\ Via Trieste 63\\ \\}
\email{treu@math.unipd.it}

\author{Davide Vittone}
\address{Davide Vittone: Dipartimento di Matematica\\Universit\`a di Padova\\ Via Trieste 63\\ \\}
\email{vittone@math.unipd.it}



\begin{document}

\begin{abstract}
We consider the area functional for $t$-graphs in the sub-Riemannian Heisenberg group and study minimizers of the associated Dirichlet problem. We prove that, under a bounded slope condition on the boundary datum, there exists a unique minimizer and that this minimizer is Lipschitz continuous. We also provide an example showing that, in the first Heisenberg group, Lipschitz regularity is sharp even under the bounded slope condition.
\end{abstract}

\maketitle

\section{Introduction}
The area functional for the {\em $t$-graph} of a function $u\in W^{1,1}(\Omega)$ in the sub-Riemannian {\em Heisenberg group} $\bbH^n\equiv 
\R^n_x\times\R^n_y\times\R^{}_t$ is
\[
\area(u):=\int_\Omega |\nabla u+ \xva|\ud\LL^{2n}\,,
\]
where $\Omega\subset\R^{n}_{x}\times \R^n_y$ is an open set and $\xva$ is the vector field
\[
    \xva(x,y):=2(-y,x),
\]
 see \cite{pauls1,CHMY, SCV} for more details. It was shown in \cite{SCV} that the natural variational setting for this functional is the space $\BV(\Omega)$ of functions with bounded variation in $\Omega$; more precisely, it was proved that the relaxed functional of $\area$ in the $L^1$-topology is
\[
\area(u):=\int_\Omega |\nabla u+ \xva|\ud\LL^{2n}+|D^su|(\Omega), \qquad u\in \BV(\Omega),
\]
where $|D^su|(\Omega)$ is the total variation in $\Omega$ of the singular part of the distributional derivative of $u$.

In this paper, we study the minimizers of $\area$ under Dirichlet boundary conditions
\[
\min\{\area(u): u\in \BV(\Omega),u_{|\partial\Omega}=\varphi\}\,,
\]
where $\Omega$ is assumed to have Lipschitz regular boundary, $u_{|\partial\Omega}$ is the trace of $u$ on $\partial\Omega$, $\varphi\in L^1(\partial\Omega,\HH^{2n-1})$ is a fixed boundary datum and $\HH^{2n-1}$ denotes the classical Hausdorff measure of dimension $(2n-1)$. In our main result, Theorem \ref{mainteo-ex6.7} below, we prove existence, uniqueness and Lipschitz regularity of minimizers assuming that $\varphi$ satisfies the so-called {\em bounded slope condition} (see e.g. \cite{Giusti2} or Section \ref{ultSec}). We also point out that Lipschitz regularity is sharp at least in the first Heisenberg group $\bbH^1$, see Example \ref{es1}.

Our interest in this problem is twofold. On the one side, it fits in
a well-established research stream about minimal-surfaces type
problems (isoperimetric problem, existence and regularity of
$\bbH$-perimeter minimizing sets, Bernstein problem, etc.) in the
Heisenberg group, for which we refer to
\cite{LR,RR2,MontiRickly,RR1,DGNAdvMath,DGN5,Monti,R3,CCMregH1,CCMregHn,HlaPauJDiffG,HurRitRos,DGNP09,SCV}.
On the other side, we attack the problem with typical tools from the Calculus of Variations, using the so called Hilbert-Haar (or
``Semi-classical'' in \cite{Giusti2}) approach. This approach has
been recently developed and renewed to study the regularity of
minimizers of functionals starting from the regularity of the
boundary datum, without taking into account neither ellipticity nor
growth conditions on the lagrangian; see e.g. \cite{CellinaBSC,BousquetACV,BouCla, BMT,
ClarkeSNS,FiaschiTreu, MTpams, MTjde, MTccm, MTacv}.\medskip

Area minimizers have been widely studied in functional
spaces with more regularity than $\BV$. The functional $\area$ has
good variational properties such as convexity and lower
semicontinuity with respect to the $L^1$ topology. On the other
hand, it is neither coercive nor differentiable.
The lack of differentiability is due to the
presence of the so called {\em characteristic points}, i.e. the set
of points on the graph of $u$ where the tangent plane to the graph
coincides with the {\em horizontal plane}. Equivalently, the set
whose projection on $\R^{2n}$ is
\[
\displaystyle {\rm Char}(u):=\{(x,y)\in\Omega:\,\nabla u(x,y)+\xva (x,y)=0\}\,.
\]
Notice that, formally, the Euler equation associated with $\area$ is
\[
\dive\frac{\nabla u+\xva}{|\nabla u +\xva|}=0\quad \text{in }\Omega,
\]
which degenerates at points in ${\rm Char}(u)$.
Clearly, the set ${\rm Char}(u)$ has a prominent role in studying minimizers' regularity. Several examples of minimizers with at most Lipschitz regularity have been provided in $\bbH^1$, see e.g. \cite{pauls2,CHY,R2}; more recently, a non-continuous minimizer was exhibited in \cite{SCV}. A variety of very interesting results can be found in \cite{pauls1,GP1,CHMY,CHw,RR2,Sh} (where a priori $C^2$ regularity is assumed for minimizers) and in \cite{CHY2,CHMY2} (for $C^1$ minimizers), also in connection with the Bernstein problem for $t$-graphs. The much more delicate case of minimizers in Sobolev spaces was attacked in \cite{CHY}, where interesting uniqueness and comparison theorems for minimizers in the space $W^{1,2}$ were proved. Uniqueness results for Sobolev minimizers are proved also in \cite{CH2}.

Concerning the existence issue, the existence of Lipschitz
minimizers for the Dirichlet problem for $\area$ was established in
\cite{CHY}, by utilizing an elliptic
approximation argument, for $C^{2,\alpha}$-smooth boundary data on $C^{2,\alpha}$-smooth
and ``parabolically-convex'' domains. We have also to mention the papers
\cite{pauls1} and
\cite{CH}, about which we will say a few words below. Notice that the existence of minimizers is in general not
guaranteed even for smooth boundary data on smooth domains, see
\cite[Example 3.6]{SCV}. Nevertheless, an existence result (in $\BV$
and for any datum $\varphi$) was proved in \cite{SCV} for a penalized functional, see also Section \ref{proparea}.
\medskip


%

In this paper, we consider the Dirichlet problem for the functional
$\area$ in the space of function with bounded variation and we study
minimizers  under the assumption that the boundary datum $\varphi$
satisfies a bounded slope condition with constant $Q$ ($Q$-B.S.C.
for short). This approach is inspired by some classical and well
known results in the Calculus of Variations that go back to Hilbert
and Haar (\cite{Hilbert} and \cite{Haar} respectively), has then
been used by Hartman, Nirenberg and Stampacchia
\cite{HartmanBSC,HartmanConvex,HartmanNirenberg,HartStamp,Stampacchia}. The main classical result in this framework (see e.g. \cite[Chapter 1]{Giusti2} or, for minimal surfaces,  \cite{Miranda})
states that, given a strictly convex integral functional depending
only on the gradient and a boundary datum satisfying the $Q$-B.S.C.,
there exists a unique minimizer in class of Lipschitz functions and
its Lipschitz constant is not greater than $Q$. We stress that, in the general
statement in \cite{Giusti2}, neither growth assumptions nor ellipticity
conditions are required whereas these hypotheses usually play a
crucial role respectively for existence and regularity results.
In the same setting,  it has been recently proved in \cite{CellinaBSC} that any continuous minimizer in the Sobolev space $W^{1,1}$ is
Lipschitz continuous with Lipschitz constant not greater than $Q$.

The main tools used in these results are: the validity of comparison
principles between minimizers, the invariance of minimizers under
translations of the domain, and a Haar-Rad\`o type theorem stating
that  the maximum among the difference quotients of the minimizer is
attained at the boundary. Subsequent papers (\cite{MTjde,MTccm})
addressed the problem of considering functionals that are not
strictly convex. The main difficulty, in this case, is that
comparison principles do not hold in their generality (an example
can be found in \cite{Cel} or in \cite{MThrcv}) and it is overcome
by detecting special minimizers which instead satisfy them, see
\cite{Cel, CellinaUNablaU, CMT, MThrcv}. It is worth remarking that, in
all these papers, there are assumptions guaranteeing the boundedness
of the faces of the epigraph of  the lagrangian. Concerning
functionals depending also on lower order terms, this approach works
for lagrangians of sum type as $f(\xi)+g(z,u)$, see \cite{BouCla,
BousquetCV, CellinaUNablaU, FiaschiTreu}.

In the present paper we use some of the techniques described above
but we encounter new difficulties that we briefly sketch here and
will be discussed in details in the following sections. First of
all, we deal with functions of bounded variation and we use ideas of
\cite{TreuBV}, where functionals depending only on the gradient are
considered. The second point is the dependance of our functional on
points of $\Omega$ encoded in the vector field $\xva$. Moreover, the
epigraph of our lagrangian has unbounded faces. All these
peculiarities led us to face many new technical problems that will
be considered in sections \ref{strumin} and \ref{uniqSec}.

In the framework of the first Heisenberg group $\bbH^1$, the bounded
slope condition was already considered in \cite{pauls1} by S. D.
Pauls. Using an approximation scheme by means of minimal surfaces in
the Riemannian approximations of $\bbH^n$, he  showed the existence
in $W^{1,p}(\Omega)\cap C^0(\Omega)$ of weak solutions  to the Euler
equation associated with $\area$.

Before stating our main result let us underline some
peculiarities about the bounded slope condition. On the one hand, it is a quite
restrictive assumption because it implies that, unless $\varphi$ is
affine, $\Omega$ is convex. On the other hand, the class of functions
satisfying it is quite large, since M. Miranda \cite{Miranda} proved that, if
$\Omega$ is uniformly convex, then any $\varphi\in\mathcal{C}^{1,1}$
satisfies the $Q$-B.S.C. for some $Q$.

\begin{teo}\label{mainteo-ex6.7}
Let $\Omega\subset\bbR^{2n}$ be open, bounded and with Lipschitz regular boundary, and let $\varphi:\partial\Omega\to\R$ satisfy the $Q$-B.S.C. for some $Q>0$. Then, the minimization problem
\begin{equation}\label{problema}
\min\big\{\area(u):
u\in \BV(\Omega),\ u_{|\partial\Omega}=\varphi\big\}
\end{equation}
admits a unique solution $\hat{u}$. Moreover, $\hat u$ is Lipschitz continuous and $\lip(\hat{u})\leq \overline Q=\overline Q(Q,\Omega)$.
\end{teo}

Notice that, if $\varphi$ satisfies the B.S.C., then it is Lipschitz
continuous on $\partial\Omega$: in this sense, our assumptions on
the boundary datum are stronger than those in \cite[Theorem A]{CH},
where the authors prove the existence and the continuity of $\BV$
minimizers on $C^{2,\alpha}$ parabolically convex domains assuming
only the continuity of the boundary datum. Nevertheless, we are able
to obtain stronger results (namely: uniqueness and Lipschitz
regularity of the minimizer) on (possibly) less regular domains. In fact our
result applies, in particular, when $\Omega$ is uniformly convex and
$\varphi$ is (the restriction to $\partial\Omega$ of) a function of
class $\mathcal{C}^{1,1}$: in this case, as previously
mentioned, $\varphi$ automatically satisfies the B.S.C..

 We conjecture that,
as \cite[Theorem A]{CH}, our main result holds as well for more general functionals.

We also want to stress a couple of interesting points concerning Theorem \ref{mainteo-ex6.7}. First, in contrast with the Semi-Classical Theory, the Lipschitz constant of the minimizer may a priori be greater than the constant $Q$ given by the bounded slope condition. Second, Theorem \ref{mainteo-ex6.7} is sharp at least in $\bbH^1$ in the sense that, even under the bounded slope condition, a minimizer might not be better than Lipschitz continuous, see Example \ref{es1}.

The proof of Theorem \ref{mainteo-ex6.7} is based on several intermediate results which have an independent interest. We mention, for instance, a Comparison Principle for minimizers, Theorem \ref{confro}, which in turn is based on the existence of the (pointwise a.e.) ``maximum'' and ``minimum'' in the family of minimizers, see Proposition \ref{esistmax}. The uniqueness in Theorem \ref{mainteo-ex6.7} is based on a criterion stated in Proposition \ref{unicitapp'} (for which we have to credit \cite{CHY}) and on the fact that affine functions are the unique minimizers under their own boundary datum, see Theorem \ref{affineMIN}.\medskip

The structure of the paper is the following. Section \ref{recallBV} contains basic facts about functions with bounded variation and their traces. In Section \ref{proparea} we recall several preliminary results about the functional $\area$. Section \ref{strumin} is devoted to the study of the set of minimizers and its structure, with particular regard to comparison principles. In Section \ref{uniqSec} we prove the uniqueness results in Proposition \ref{unicitapp'} and Theorem \ref{affineMIN}. Finally, Section \ref{ultSec} is devoted to the proof of Theorem \ref{mainteo-ex6.7}.

\section{Functions of Bounded Variation and traces}\label{recallBV}

Aim of this section is to recall some basic properties of the space of functions of bounded variation; we refer to the monographs \cite{ambfuspal,Giusti} for a more extensive account on the subject as well as for proofs of the results we are going to recall here.

Let $\Omega$  be an open set in $\bbR^n$. We say that $u\in L^1(\Omega)$ has {\em bounded variation} in $\Omega$ if
\begin{equation}\label{ilsup}
\sup\Big\{\int_{\Omega} u\ \dive \varphi\ \ud\LL^n:\varphi\in C^1_c(\Omega), \norm{\varphi}\leq 1\Big\}<+\infty;
\end{equation}
equivalently, if there exist a $\R^n$-valued Radon measure $Du:=(Du_1,\dots, Du_n)$ in $\Omega$ which represents the distributional derivative of $u$, i.e., if
\[
\int_{\Omega} u\frac{\partial\varphi}{\partial x_i}\ud\LL^n=-\int_{\Omega} \varphi\  \ud D_i u\quad \forall \varphi\in C^1_c(\Omega),\ \forall i=1,\ldots, n.
\]
The space of functions with bounded variation in $\Omega$ is denoted by $\BV(\Omega)$. By definition, $W^{1,1}(\Omega)\subset \BV(\Omega)$ and $Du=\nabla u\,\LL^n$ for any $u\in W^{1,1}(\Omega)$.

We denote by $|Du|$ the total variation of the measure $Du$; $|Du|$ defines a finite measure on $\Omega$ and the supremum in \eqref{ilsup} coincides with $|Du|(\Omega)$.
It is well-known that $\BV(\Omega)$ is a Banach space when endowed with the norm
\[
    \norm{u}_{\BV}:=\norm{u}_{L^1}+ |Du|(\Omega).
\]

By the Radon-Nikodym Theorem, if $u\in \BV(\Omega)$ one can write $Du=D^au+D^su$, where $D^au$ is the absolutely continuous part of $Du$ with respect to $\LL^n$ and $D^su$ is the singular part of $Du$ with respect to $\LL^n$. We denote by $\nabla u\in L^1(\Omega)$ the density of $D^au$ with respect to $\LL^n$, so that $D^au=\nabla u\,\LL^n$. It turns out that, if $u\in \BV(\Omega)$, then $u$ is {\em approximately differentiable} at a.e. $x\in\Omega$ with approximate differential $\nabla u(x)$, i.e.,
\[
    \lim_{\rho\to 0^+}\fint_{B(x,\rho)} \frac{|u(y)-\tilde u(x)-\left\langle \nabla u(x), y-x\right\rangle|}{\rho}\ \ud\LL^n=0\qquad\text{for $\LL^n$-a.e. }x\in\Omega\,.\medskip
\]
%

%

\medskip
We now recall a few basic facts about boundary trace properties for BV functions. Assume  that $\Omega\subset\R^n$ is a bounded open set with Lipschitz regular boundary; the spaces $L^p(\partial\Omega), p\in[1,+\infty]$, will be always understood with respect to the (finite) measure $\HH^{n-1}\res\partial\Omega$, where $\HH^{n-1}$ denotes the $(n-1)$-dimensional {\em Hausdorff measure} on $\R^n$ (see again \cite{ambfuspal} or \cite{Giusti}). It is well-known that for any $u\in \BV(\Omega)$ there exists a (unique) function $u_{|\partial\Omega}\in L^1(\partial\Omega)$ such that, for $\HH^{n-1}$-a.e. $x\in \partial\Omega$,
\[
\lim_{\rho\to0^+}\rho^{-n} \int_{\Omega\cap B(x,\rho)}|u-u_{|\partial\Omega}(x)|\ud\LL^n=\lim_{\rho\to0^+}\fint_{\Omega\cap B(x,\rho)}|u-u_{|\partial\Omega}(x)|\ud\LL^n=0\,.
\]
The function $u_{|\partial\Omega}$
 is called the {\em trace} of $u$ on $\partial\Omega$. The trace operator $u\mapsto u_{|\partial\Omega}$ is linear and continuous between $(\BV(\Omega),\norm\cdot_{\BV})$ and $L^1(\partial\Omega)$; actually, it is continuous also when $\BV(\Omega)$ is endowed with the (weaker) topology induced by the so-called {\em strict convergence}, see \cite[Definition 3.14]{ambfuspal}.

\begin{oss}\label{traccesupinf}{\rm
It is well-known that, if $u_1,u_2\in \BV(\Omega)$, then $\overline u:=\max\{u_1,u_2\}$ and $\underline u:=\min\{u_1,u_2\}$ belong to $\BV(\Omega)$; moreover, one can show that
\[
\overline u_{|\partial\Omega}=\max\{u_{1|\partial\Omega},u_{2|\partial\Omega}\},\qquad \underline u_{|\partial\Omega}=\min\{u_{1|\partial\Omega},u_{2|\partial\Omega}\}\,.
\]
The proof of this fact follows in a standard way from the very definition of traces.
}\end{oss}

Since $Du\ll|Du|$ we can write $Du=\sigma_u|Du|$ for a $|Du|$-measurable function $\sigma_u:\Omega\to\bbS^{n-1}$. With this notation, one has also
\begin{equation}\label{defvartr}
\int_\Omega u\,\dive\varphi\ud\LL^n = -\int_\Omega \langle \sigma_u,\varphi\rangle \ud |Du| + \int_{\partial\Omega} u_{|\partial\Omega}\, \langle\varphi,\nu_\Omega\rangle \ud\HH^{n-1}\qquad\forall \varphi\in C^1_c(\R^n,\R^n)
\end{equation}
where $\nu_\Omega$ is the unit outer normal to $\partial\Omega$.

Finally, we recall the following fact, whose proof stems from \eqref{defvartr}.

\begin{prop}[{\cite[Remark 2.13]{Giusti}}]\label{tracceatratti}
Assume that $\Omega$ and $\Omega_0$ are open subsets of $\R^n$ with bounded Lipschitz boundary and such that $\Omega\Subset\Omega_0$. If $u\in \BV(\Omega)$ and $v\in \BV(\Omega_0\setminus\overline\Omega)$, then the function
\[
f(x):=
\left\{\begin{array}{ll}
u(x) & \text{if }x\in\Omega\\
v(x) & \text{if }x\in\Omega_0\setminus\overline\Omega
\end{array}\right.
\]
belongs to $\BV(\Omega_0)$ and
\[
|Df|(\partial\Omega)= \int_{\partial\Omega}|u_{|\partial\Omega}-v_{|\partial\Omega}|\ud\HH^{n-1}\,,
\]
where we have used the notation $v_{|\partial\Omega}$ to mean $(v_{|\partial(\Omega_0\setminus\overline\Omega)})\res{\partial\Omega}$.
\end{prop}

\section{The area functional for \texorpdfstring{$t$}{t}-graphs in the Heisenberg group}\label{proparea}
Before introducing the area functional $\area$ with more details, we need some preliminary notation. For $z:=(x,y)\in\R^{n}\times\R^n$ we define
\[
z^\ast:=(-y,x)\in\R^{2n}\,.
\]
Let us state some useful property of the map $z\mapsto z^\ast$.

\begin{lem}\label{starop}
The following properties hold:
\begin{enumerate}
    \item[(i)] if $z_1,z_2\in\bbR^{2n}$ are linearly dependent, then $z_1\cdot z_2^*=0$;
    \item[(ii)] $z_1\cdot z_2=z_1^*\cdot z_2^*$ for each $z_1,z_2\in\bbR^{2n}$;
    \item[(iii)] if $\Omega\subset\R^{2n}$ is open and $f\in C^{\infty}(\Omega)$, then $\dive(\nabla f)^*=0$ on $\Omega$.
\end{enumerate}
\end{lem}
\begin{proof}
The first two statements are straightforward. To prove (iii), observe that
$(\nabla f)^*=(-\partial_{n+1} f,\ldots, -\partial_{2n} f, \partial_1 f,\ldots, \partial_n f)$, thus
\[
  \dive(\nabla f)^*=-\sum_{i=1}^n \partial_{i}\partial_{n+i} f+\sum_{i=1}^n \partial_{n+i}\partial_i f = 0\,.
\]
\end{proof}


Given an open set $\Omega\subset\R^{2n}$ we define the convex functional $\area_\Omega:\BV(\Omega)\to\R$
\[
\area_\Omega(u):=\int_\Omega |\nabla u+ \xva|\ud\LL^{2n}+|D^su|(\Omega)\,,
\]
where $\xvett(z):=2z$ and, clearly, $\xva(z)=2z^*$.

When the open set $\Omega$ is clear from the context, we will simply write $\area$ instead of $\area_\Omega$. Using the standard identification of the Heisenberg group $\bbH^n$ with $\R^{2n}_z\times\R_t$, there holds
\begin{equation}\label{areaperim}
\area(u)=|\partial E^t_u|_\bbH(\Omega\times\R)\,,
\end{equation}
where $|\partial E^t_u|_\bbH(\Omega\times\R)$ denotes the {\em $\bbH$-perimeter} in $\Omega\times\R\subset\bbH^n$ of the $t$-subgraph
\[
E^t_u:=\{(z,t)\in\bbH^n: z\in\Omega,\ t<u(z)\}
\]
of $u$. See \cite{SCV} for more details. It was proved in \cite{SCV} that $\area$ is lower semicontinuous with respect to the $L^1$-convergence and
\[
\area(u)=\inf \left\{\liminf_{j\to\infty} \int_\Omega |\nabla u_j+ \xva|\ud\LL^{2n}: (u_j)_{j\in\bbN}\subset C^1(\Omega),u_j\to u\text{ in } L^1(\Omega)\right\}\,.
\]

The following approximation result holds

\begin{prop}\label{appros}
Let $\Omega\subset\bbR^{2n}$ be a bounded open set with Lipschitz boundary. Let $u\in \BV(\Omega)$ with $u_{|\partial\Omega}=\varphi\in L^1(\Omega)$; then there exists a sequence $(u_k)_k\subset C^{\infty}(\Omega)$ converging to $u$ in $L^1(\Omega)$ and such that
\begin{align}
    &(u_{k})_{|\partial\Omega}=\varphi\quad \forall k\in\bbN,\label{7777}\\
    &\area_{\Omega}(u)=\lim_{k\to\infty}\int_{\Omega}|\nabla u_k+\xva|\ud\LL^{2n}.\label{8888}
\end{align}
\end{prop}
\begin{proof}
The existence of a sequence $(u_k)_k\subset C^{\infty}(\Omega)$ converging to $u$ in $L^1(\Omega)$ and such that \eqref{8888} holds was proved in \cite[Theorem 3.2]{SCV}, see also \cite[Corollary 3.3]{SCV}. More precisely, this sequence was constructed in Step 4 of the proof of \cite[Theorem 3.2]{SCV} by imposing certain conditions on suitably mollified functions, see formulae (3.5)-(3.7) therein. Reasoning as in \cite[Remark 1.18]{Giusti}, it can be proved that condition (3.5) of \cite{SCV} implies that
\[
\lim_{\rho\to 0^+} \rho^{-n} \int_{\Omega\cap B(z,\rho)}|u-u_k|\ud\LL^{2n}=0\quad\text{for $\HH^{2n-1}$-a.e. }z\in\partial\Omega
\]
and \eqref{7777} follows from the definition of traces.
\end{proof}

We are interested in the existence of minimizers for $\area$ under prescribed boundary conditions. Assuming $\Omega$ to be a bounded domain with Lipschitz boundary, we consider
\begin{align}
    M_1:=\inf\{\area(u):\,u\in \BV(\Omega),u_{|\partial\Omega}=\,\varphi\}.
\end{align}
It is known that the infimum $M_1$ might not be attained even with $\Omega$ and $\varphi$ smooth, see \cite[Example 3.6]{SCV}.

On the other hand, one can consider the functional
\[
\area_{\varphi,\Omega}(u):=\area_\Omega(u) + \int_{\partial\Omega}|u_{|\partial\Omega}-\varphi|\ud\HH^{2n-1}\,
\]
where the integral on the right hand side can be seen as a penalization for $u$ not taking the boundary value $\varphi$; this penalization is natural from the viewpoint of the geometry of $\bbH^n$ as shown in \cite[Remark 3.8]{SCV}. Again,  we will simply write $\area_\varphi$ instead of $\area_{\varphi,\Omega}$ when the open set $\Omega$ is clear from the context. By using the Direct Method of the Calculus of Variations (see again \cite{SCV}), it can be shown that the problem
\[
    M_2:=\min\{\area_{\varphi}(u):\,u\in \BV(\Omega)\}
\]
admits always a solution.

Let us show that the Lavrentiev phenomenon does not occur for our minimization problem.

\begin{prop}
Let $\Omega\subset\bbR^{2n}$ be a bounded open set with Lipschitz regular boundary and $\varphi$ be in $L^1(\partial\Omega)$; then, setting
\[
M_3:=\inf\{\area(u):\,u\in C^\infty(\Omega)\cap W^{1,1}(\Omega),u_{|\partial\Omega}=\,\varphi\}
\]
we have $M_1=M_2=M_3$, where $M_1$ and $M_2$ are defined above.
\end{prop}
\begin{proof}
Clearly, one has $M_3\ge M_1\ge M_2$ because
\[
\{u\in C^{\infty}(\Omega)\cap W^{1,1}(\Omega): u_{|\partial\Omega}=\,\varphi\}\subset\{u\in \BV(\Omega):u_{|\partial\Omega}=\,\varphi\}\subset \BV(\Omega)
\]
and $\area_{\varphi}$ coincides with $\area$ on $\{u\in \BV(\Omega):u_{|\partial\Omega}=\,\varphi\}$.
Let $u\in\BV(\Omega)$ with $u_{|\partial\Omega}=\varphi$ and consider $(u_k)_k\subset C^{\infty}(\Omega)$ as in Proposition \ref{appros}. Then
\[
M_3\leq \area(u) = \lim_{k\to\infty}\area(u_k),
\]
which implies $M_3\leq M_1$ and hence $M_3=M_1$. Finally, the equality $M_1=M_2$ has been established in {\cite[Theorem 1.4]{SCV}}, and the proof is accomplished.
\end{proof}

\begin{oss}{\rm
One can easily show that, if the boundary datum $\varphi$ is Lipschitz continuous on $\partial\Omega$, then the equalities
\[
M_1=M_2=M_3=\inf\{\area(u):\,u\in \lip(\Omega),u_{|\partial\Omega}=\,\varphi\}
\]
hold.
}\end{oss}

\section{The set of minimizers and Comparison Principles}\label{strumin}

The aim of this section is to establish a Comparison Principle for
minimizers of the area functional with the penalization on the boundary. It
is well-known that Comparison Principles are strictly related with
uniqueness of solutions and that functionals defined in $\BV$ do not
have, in general, uniqueness of minimizers, even in the case of a
strictly convex lagrangian. In our case, the lagrangian
$f(z,\xi):=|\xi+\xva(z)|$ is not even strictly convex. The validity of
Comparison Principles for non strictly convex functional has been
studied in \cite{Cel,CMT,Mari,MTccm,MThrcv} in the case of superlinear
growth, and it has been proved for special classes of minimizers. In
this section we follow the same ideas, but we have to overcome some
new difficulties that are related both with the properties of $\BV$
and with the fact  that the lagrangian depends also on the variable $z$.

The main Comparison Principle is stated in Theorem \ref{confro}. Its proof is based on several steps and, in particular, it relies on some inequalities that, in our opinion, have an interest on their own. For this reason, we enunciate them as separate propositions. The proof of Theorem \ref{confro} will then follow in a few lines.

We remark that similar results for functionals with linear growth, depending just on the gradient and defined in the space of function of bounded variation, have been recently obtained in \cite{TreuBV}.

The next two propositions state two inequalities between the values of the area functional at $u_1$, $u_2$, $u_1\vee u_2$ and $u_1\wedge u_2$. The first one is stated for the area functional and the second one is for the functional with the penalization on the boundary. We observe that, when one deals with integral functionals defined in Sobolev spaces, these inequalities turn out to be equalities, whose proof is straightforward (see Lemma 5.1 in \cite{MThrcv}).

\begin{prop}\label{stimaMIN}
Let $\Omega\subset\bbR^{2n}$ be a bounded open set with Lipschitz
boundary. Let $u_1,u_2\in \BV(\Omega)$ be fixed. Then
\begin{align}\label{ineqfin}
    \area(u_1\vee u_2)+\area(u_1\wedge u_2)\leq \area(u_1)+\area(u_2)\,.
\end{align}
\end{prop}
%
%
%

\begin{proof}
Recalling \eqref{areaperim},  inequality \eqref{ineqfin} is equivalent to
\[
|\partial E^t_{u_1\vee u_2}|_\bbH(\Omega\times\R) + |\partial E^t_{u_1\wedge u_2}|_\bbH(\Omega\times\R) \leq |\partial E^t_{u_1}|_\bbH(\Omega\times\R) +|\partial E^t_{u_2}|_\bbH(\Omega\times\R)
\]
which follows from \cite[Proposition 2.3]{SCV} on noticing that
\[
E^t_{u_1\vee u_2} = E^t_{u_1}\cup E^t_{u_2},\quad E^t_{u_1\wedge u_2} = E^t_{u_1}\cap E^t_{u_2}\,.
\]
\end{proof}

The next Proposition is the analogue of the previous one for the
functional $\area_\varphi$ where the boundary conditions are
taken into account.

\begin{prop}\label{nonpaga}
Let $\Omega\subset\bbR^{2n}$ be a bounded open set with Lipschitz regular boundary. Then, for each $u_1,u_2\in \BV(\Omega)$ and $\varphi_1,\varphi_2\in L^1(\partial\Omega)$ we have
\begin{equation}\label{eqnonpaga}
    \area_{\varphi_1\vee \varphi_2,\Omega}(u_1\vee u_2)+\area_{\varphi_1\wedge \varphi_2,\Omega}(u_1\wedge u_2)\leq \area_{\varphi_1,\Omega}(u_1)+\area_{\varphi_2,\Omega}(u_2)
\end{equation}
\end{prop}
\begin{proof}
Let us fix a bounded open set $\Omega_0\subset \R^{2n}$ with
Lipschitz boundary and such that $\Omega\Subset\Omega_0$. By
\cite[Theorem 2.16]{Giusti} it is possible to find $ f_1$,$ f_2$ in
$ W^{1,1}(\Omega_0\setminus\overline\Omega)$ such that
\[
f_{1|\partial \Omega}=\varphi_1\quad\text{and}\quad f_{2|\partial \Omega}=\varphi_2\,.
\]
Define
\[
v_1:=
\left\{
\begin{aligned}
&u_1\quad \mbox{in}\ \Omega\\
&f_1\quad \mbox{in}\ \Omega_0\setminus\overline\Omega,
\end{aligned}
\right.\qquad
v_2:=
\left\{
\begin{aligned}
&u_2\quad \mbox{in}\ \Omega\\
&f_2\quad \mbox{in}\ \Omega_0\setminus\overline\Omega
\end{aligned}
\right.
\]
so that
\[
v_1\vee v_2= \left\{
\begin{aligned}
&u_1\vee u_2\quad \mbox{in}\ \Omega\\
&f_1\vee f_2\quad \mbox{in}\ \Omega_0\setminus\overline\Omega,
\end{aligned}
\right.\qquad v_1\wedge v_2= \left\{
\begin{aligned}
&u_1\wedge u_2\quad \mbox{in}\ \Omega\\
&f_1\wedge f_2\quad \mbox{in}\ \Omega_0\setminus\overline\Omega.
\end{aligned}
\right.
\]
We have $v_1,v_2, v_1\vee v_2,v_1\wedge v_2\in \BV(\Omega_0)$ and
Lemma \ref{stimaMIN} gives
\begin{equation}\label{vvv}
\area_{\Omega_0}(v_1\vee v_2)+\area_{\Omega_0}(v_1\wedge v_2)\leq
\area_{\Omega_0}(v_1)+\area_{\Omega_0}(v_2).
\end{equation}
Writing $(f_1\vee f_2)_{|\partial\Omega}$ for $((f_1\vee
f_2)_{|\partial(\Omega_0\setminus\overline\Omega)})\res{\partial\Omega}$,
we have by Proposition \ref{tracceatratti} and Remark \ref{traccesupinf}
\begin{equation}\label{omega0omega}
\begin{split}
& \area_{\Omega_0}(v_1\vee v_2)\\
= & \int_{\Omega_0}|\nabla(v_1\vee v_2)+\xva|\ud\LL^{2n} + |D^s(v_1\vee v_2)|(\Omega_0)\\
= & \area_{\Omega}(v_1\vee v_2) + \area_{\Omega_0\setminus\overline\Omega}(v_1\vee v_2) + |D^s(v_1\vee v_2)|(\partial\Omega)\\
= & \area_{\Omega}(u_1\vee u_2) +
\area_{\Omega_0\setminus\overline\Omega}(f_1\vee f_2) +
\int_{\partial\Omega}
|(u_1\vee u_2)_{|\partial\Omega} - (f_1\vee f_2)_{|\partial\Omega}|\ud\HH^{2n-1}\\
= & \area_{\Omega}(u_1\vee u_2) +
\int_{\Omega_0\setminus\overline\Omega}|\nabla(f_1\vee
f_2)+\xva|\ud\LL^{2n}\\
&+ \int_{\partial\Omega}|(u_1\vee u_2)_{|\partial\Omega} - (\varphi_1\vee \varphi_2)|\ud\HH^{2n-1}\\
=& \area_{\varphi_1\vee \varphi_2,\Omega} (u_1\vee u_2) +
\int_{\Omega_0\setminus\overline\Omega}|\nabla(f_1\vee
f_2)+\xva|\ud\LL^{2n}
\end{split}
\end{equation}
where we also utilized Remark \ref{traccesupinf}. In a similar way one obtains
\begin{equation}\label{omega0omega1}
\begin{split}
& \area_{\Omega_0}(v_1\wedge v_2) =
\area_{\varphi_1\wedge\varphi_2,\Omega}((u_1\wedge u_2)) +
\int_{\Omega_0\setminus\overline\Omega}|\nabla(f_1\wedge f_2)+\xva|\ud\LL^{2n}\\
& \area_{\Omega_0}(v_1) = \area_{\varphi_1,\Omega}(u_1) + \int_{\Omega_0\setminus\overline\Omega}|\nabla f_1+\xva|\ud\LL^{2n} \\
& \area_{\Omega_0}(v_2) = \area_{\varphi_2,\Omega}(u_2) + \int_{\Omega_0\setminus\overline\Omega}|\nabla f_2+\xva|\ud\LL^{2n} \,.
\end{split}
\end{equation}
Now, \eqref{eqnonpaga} will follow from \eqref{vvv}, \eqref{omega0omega} and \eqref{omega0omega1} provided we show that
\[
\begin{split}
& \int_{\Omega_0\setminus\overline\Omega}|\nabla(f_1\vee f_2)+\xva|\ud\LL^{2n} +
\int_{\Omega_0\setminus\overline\Omega}|\nabla(f_1\wedge f_2)+\xva|\ud\LL^{2n}\\
= & \int_{\Omega_0\setminus\overline\Omega}|\nabla
f_1+\xva|\ud\LL^{2n} +
\int_{\Omega_0\setminus\overline\Omega}|\nabla
f_2+\xva|\ud\LL^{2n}\,.
\end{split}
\]
This can be easily seen by utilizing the well-known fact that 
\[
\begin{split}
& \nabla(f_1\vee f_2) = (\nabla f_1)\chi_{\{f_1\geq f_2\}} + (\nabla f_2)\chi_{\{f_1<f_2\}},\\
& \nabla(f_1\wedge f_2) = (\nabla f_2) \chi_{\{f_1\geq f_2\}} + (\nabla f_1)\chi_{\{f_1<f_2\}}\,.
\end{split}
\]
The proof is accomplished.
\end{proof}

Given a  bounded open set $\Omega\subset\bbR^{2n}$ with Lipschitz regular boundary and a function $\varphi\in L^1(\partial\Omega)$ we define
\[
\cc_\varphi:=\argmin\limits_u \area_{\varphi,\Omega}(u)=\argmin\limits_u \Big\{\area_\Omega(u)+\int_{\partial\Omega}|u-\varphi|\ud\HH^{2n-1}\Big\}\,.
\]
The set $\cc_\varphi\subset \BV(\Omega)$ is not empty by
\cite[Theorem $1.4$]{SCV}.

\begin{cor}\label{maxmin}
Let $\varphi_1,\varphi_2 \in L^1(\partial\Omega)$ be such that
$\varphi_1\leq \varphi_2$ $\HH^{2n-1}$-a.e. on $\partial\Omega$ and assume
that $u_1\in\cc_{\varphi_1}$ and $u_2\in\cc_{\varphi_2}$. Then
$(u_1\vee u_2)\in\cc_{\varphi_2}$ and  $(u_1\wedge
u_2)\in\cc_{\varphi_1}\,$.
\end{cor}
\begin{proof}
The assumptions that $u_2$ is a minimizer of $\area_{\varphi_2,\Omega}$ and $\varphi_1\le\varphi_2$ imply that
\begin{equation}\label{veemax}
\area_{\varphi_2,\Omega}(u_1\vee u_2)\geq
\area_{\varphi_2,\Omega}(u_2)\,.
\end{equation}
Analogously we have
\begin{equation}\label{wedgemax}
\area_{\varphi_1,\Omega}(u_1\wedge u_2)\geq
\area_{\varphi_1,\Omega}(u_1)\,.
\end{equation}
By Proposition \ref{nonpaga} it follows that
\[
\area_{\varphi_1,\Omega}(u_1\vee
u_2)+\area_{\varphi_2,\Omega}(u_1\wedge
u_2)=\area_{\varphi_1,\Omega}(u_1)+\area_{\varphi_2,\Omega}( u_2)
\]
so that equality holds both in \eqref{veemax} and in \eqref{wedgemax}.
\end{proof}
In \cite{MTccm} it has been proved that the set of
minimizers of a superlinear convex functional has a {\emph maximum}
$\overline u$ (resp. a minimum $\underline u$) defined as the
pointwise supremum (infimum) of the minimizers. These special
minimizers are then used to prove one-sided Comparison Principles.
Now, with a different technique required by the use of function of
bounded variation, we prove a similar result.

\begin{prop}\label{esistmax}
Let $\Omega\subset\bbR^{2n}$ be a bounded open set with Lipschitz regular boundary and let $\varphi\in L^1(\partial\Omega)$. Then, there
exists $\overline{u},\underline{u}\in \cc_\varphi$ such that the
inequalities
\begin{equation}\label{min9}
    \underline u \leq u\leq \overline u\quad \LL^{2n}\mbox{-a.e. in}\  \Omega
\end{equation}
hold for any $u\in\cc_\varphi$.
\end{prop}
\begin{proof}
We start by proving that $\cc_\varphi$ is bounded in $\BV(\Omega)$. Given $u\in\cc_\varphi$, we define $J:=\area_{\varphi,\Omega}(u)<\infty$; clearly, $J$ depends only on $\varphi$ and not on $u$. We have
\begin{equation}\label{stimavar}
\begin{split}
    |D u|(\Omega) = & \int_\Omega |\nabla u|\ud\LL^{2n} + |D^su|(\Omega)\\
    \leq & \int_\Omega |\nabla u+\xva|\ud\LL^{2n} + \int_\Omega |\xva|\ud\LL^{2n} + |D^su|(\Omega) + \int_{\partial\Omega}|u-\varphi|\ud\HH^{2n-1}\\
    = & J+\int_\Omega |\xva|\ud\LL^{2n}<\infty\,.
\end{split}
\end{equation}
Moreover, by \cite[Theorem 1.28 and Remark 2.14]{Giusti} there exists $c=c(n)>0$ such that
\[
\begin{split}
\norm{u}_{L^1(\Omega)} &\leq |\Omega|^{1/2n}\norm{u}_{L^{2n/(2n-1)}(\Omega)}\\
&\leq  c\,|\Omega|^{1/2n}\Big( |D u|(\Omega)+\int_{\partial\Omega}|u_{|\partial\Omega}|    \ud\HH^{2n-1}\Big)\\
&\leq   c|\Omega|^{1/2n}\Big( |D u|(\Omega)+\int_{\partial\Omega}|u_{|\partial\Omega}-\varphi|\ud\HH^{2n-1}  +\int_{\partial\Omega}|\varphi|\ud\HH^{2n-1}\Big)\\
&= c|\Omega|^{1/2n}\Big( J+\int_{\partial\Omega}|\varphi|\ud\HH^{2n-1}\Big)\,.
\end{split}
\]
where $|\Omega|:=\LL^{2n}(\Omega)$. This, together with \eqref{stimavar}, implies that $\cc_\varphi$ is bounded in $\BV(\Omega)$.

Therefore (see \cite[Theorem 3.23]{ambfuspal}), $\cc_\varphi$ is pre-compact in $L^1(\Omega)$, i.e., for every sequence $(u_h)_{h\in\bbN}\subset \cc_\varphi$ there exist $u\in \BV(\Omega)$ and a subsequence $(u_{h_k})_{k\in\bbN}$ such that $u_{h_k}\to u$ in $L^1(\Omega)$. Since $\area_{\varphi,\Omega}$ is lower semicontinuous with respect to the $L^1$-convergence we have also
\[
\area_{\varphi,\Omega}(u)\leq\liminf_{k\to\infty} \area_{\varphi,\Omega}(u_{h_k})=J.
\]
We deduce that $u\in\cc_\varphi$, i.e., that $\cc_\varphi$ is indeed compact in $L^1(\Omega)$. Now, the functional
\[
    \BV(\Omega)\ni u\longmapsto I(u):=\int_{\Omega} u\ \ud\LL^{2n}
\]
is continuous in $L^1(\Omega)$, hence it admits maximum $\overline{u}$ and minimum $\underline u$ in $\cc_\varphi$:
let us prove that $\overline u,\underline u$ satisfy \eqref{min9} for any  $u\in \cc_\varphi$.

Assume by contradiction there exists $u\in\cc_\varphi$ such that $\Omega':=\{z\in\Omega: u(z)>\overline{u}(z)\}$ has strictly positive measure. Then, by Corollary \ref{maxmin}, $u\vee\overline u$ is in $\cc_\varphi$.
Moreover
\[
\int_\Omega (u\vee\overline u)\ud\LL^{2n}=\int_{\Omega'} u\ud\LL^{2n}+\int_{\Omega\setminus\Omega'} \overline u\ud\LL^{2n}>\int_\Omega \overline u\ud\LL^{2n}
\]
yielding a contradiction. The fact that $u\geq \underline u$ follows in a similar way.
\end{proof}

Now we can state a Comparison Principle inspired by the results obtained in \cite{MTccm} for superlinear functionals in  Sobolev spaces.

\begin{teo}\label{confro}
Let $\Omega\subset\bbR^{2n}$ be a bounded open set with Lipschitz regular
boundary; let $\varphi,\psi\in L^1(\partial\Omega)$ be such that
$\varphi\leq \psi$ $\HH^{2n-1}$-a.e. on $\partial\Omega$. Consider
the functions $\overline u,\,\underline u\in \cc_{\varphi}$ and
$\overline w,\, \underline w\in \cc_{\psi}$ such that\footnote{The existence of $\overline u,\,\underline u,\,\overline w,\,\underline w$ is guaranteed by Proposition \ref{esistmax}.}
\begin{equation}\label{uuww}
\begin{array}{ll}
    \underline u \leq u\leq \overline u\qquad &\LL^{2n}\mbox{-a.e. in }\Omega,\ \forall u\in\cc_\varphi\\
    \underline w \leq w\leq \overline w\quad &\LL^{2n}\mbox{-a.e. in }\Omega,\ \forall w\in\cc_\psi\,.
\end{array}
\end{equation}
Then
\begin{equation}\label{AM}
\overline u \leq\overline w\quad\text{and}\quad\underline
u\leq\underline w\quad \LL^{2n}\mbox{-a.e. in}\ \Omega
\end{equation}
and, in particular,
\[
\begin{split}
    & u\leq \overline w\quad \LL^{2n}\mbox{-a.e. in}\ \Omega,\ \forall u\in \cc_{\varphi}\\
    & \underline u\leq w\quad \LL^{2n}\mbox{-a.e. in}\ \Omega,\ \forall w\in \cc_{\psi}.
\end{split}
\]
\end{teo}
\begin{proof}
We have proved in Corollary \ref{maxmin} that $\overline
w\vee\overline u$ is a minimizer of $\area_{\psi}$ and
$\overline w\wedge \overline u$ is a minimizer of $\area_{\varphi}$. Assumption \eqref{uuww} then gives \eqref{AM}, which allows to conclude.
\end{proof}

The next result is a consequence of the Comparison Principle. We state it here explicitly since, in this formulation, it will be useful in the sequel.

\begin{cor}\label{valbordo}
 Let $\Omega\subset\bbR^{2n}$ be a bounded open set with Lipschitz regular boundary and $\varphi,\psi\in L^{\infty }(\partial\Omega)$; let $\overline{u},\underline{u}\in \cc_{\varphi}$ and $\overline{w},\underline{w}\in \cc_{\psi}$ be as in \eqref{uuww}.
Then, for every $\alpha\in\bbR$, one has
\begin{equation}\label{traslaz}
\begin{split}
&\overline{u}+\alpha, \underline{u}+\alpha\in \cc_{\varphi+\alpha}\\
&\underline{u}+\alpha\leq u\leq \overline{u}+\alpha\quad \LL^{2n}\mbox{-a.e. in}\ \Omega,\ \forall u\in\cc_{\varphi+\alpha}
\end{split}
\end{equation}
and
\begin{equation}\label{sup}
\begin{split}
& \norm{\overline{u}-\overline{w}}_{L^{\infty}(\Omega)}\leq\norm{\varphi-\psi}_{L^{\infty}(\partial\Omega)}\\
& \norm{\underline{u}-\underline{w}}_{L^{\infty}(\Omega)}\leq\norm{\varphi-\psi}_{L^{\infty}(\partial\Omega)}.
\end{split}
\end{equation}
In particular, the implications
\begin{equation}\label{eqrevin}
\begin{split}
& \overline u_{|\partial\Omega}=\varphi,\ \overline w_{|\partial\Omega}=\psi\quad \Rightarrow\quad \norm{\overline{u}-\overline{w}}_{L^{\infty}(\Omega)}=\norm{\varphi-\psi}_{L^{\infty}(\partial\Omega)}\\
& \underline u_{|\partial\Omega}=\varphi,\ \underline w_{|\partial\Omega}=\psi\quad \Rightarrow\quad \norm{\underline{u}-\underline{w}}_{L^{\infty}(\Omega)}=\norm{\varphi-\psi}_{L^{\infty}(\partial\Omega)}.
\end{split}
\end{equation}
hold.
\end{cor}
\begin{proof}
The statements in \eqref{traslaz}  follow at once on noticing that
\[
\area_{\varphi+\alpha,\Omega}(u+\alpha)=\area_{\varphi,\Omega}(u)\quad\forall\: u\in \BV(\Omega)\,.
\]
Let  $\alpha:=\norm{\varphi-\psi}_{L^{\infty}(\partial\Omega)}\in\bbR$, then
\[
\varphi\leq \psi+\alpha\quad \HH^{2n-1}\mbox{-a.e. in}\ \partial\Omega,
\]
and, by \eqref{traslaz} and Corollary \ref{confro}, we get
\[
\overline{u}\leq \overline{w}+\alpha\text{ and }\underline{u}\leq \underline{w}+\alpha\qquad \LL^{2n}\mbox{-a.e. in }\Omega.
\]
An analogous argument shows that
\[
\overline{w}\leq \overline{u}+\alpha\text{ and }\underline{w}\leq \underline{u}+\alpha\qquad \LL^{2n}\mbox{-a.e. in }\Omega,
\]
whence \eqref{sup}.

If the assumptions in \eqref{eqrevin} are satisfied, classical properties of traces ensure that the reverse inequalities in \eqref{sup} holds, and this gives the validity of the implications in \eqref{eqrevin}.
\end{proof}

\section{Uniqueness of special minimizers}\label{uniqSec}

This section is devoted to some uniqueness results for minimizers of the area functional. We have already recalled that, in general, minimizers of functionals defined in $\BV$ are not unique. Comparison principles are particularly interesting in this context. If we consider
the functional with the penalization on the boundary, whenever we
detect a special boundary datum yielding uniqueness of the minimizer
we also know that this minimizer satisfies the Comparison Principle,
so that it can be used as a `barrier'. We emphasize
this fact in Corollary \ref{cpaffine}, that will be the key point in
the proof of the main result of this paper.

The following uniqueness result can be proved on combining Theorems 5.1, 5.2 and 5.3 in \cite{CHY}. For the reader's benefit, we give here a slightly simplified proof.

\begin{prop}\label{unicitapp'}
Let $\Omega\subset\bbR^{2n}$ be a bounded open set with Lipschitz regular boundary; fix $p\in[1,2]$ and set $p':=\tfrac{p}{p-1}\in[2,+\infty]$. Let $\varphi\in W^{1,p'}(\Omega)$ be fixed and
consider the  minimization problem
\begin{align}\label{minW11}
    \min\{\area(u): u\in \varphi+W^{1,p}_0(\Omega)\}\,.
\end{align}
If $u\in W^{1,p'}(\Omega)$ and $v\in W^{1,p}(\Omega)$ are minimizers of \eqref{minW11}, then
$$
u=v\quad \LL^{2n}\text{-a.e. in } \Omega.
$$
\end{prop}
\begin{proof}
Let us consider the function $(u+v)/2\in \varphi+ W^{1,p}_0(\Omega)$; we claim that
\begin{align}\label{linA}
    \area\left(\tfrac{u+v}{2}\right)=\tfrac 12(\area(u)+\area(v)).
\end{align}
Indeed, the convexity of $\area$ gives
\[
        \area\left(\tfrac{u+v}{2}\right)\leq \tfrac 12(\area(u)+\area(v)),
\]
while the reverse inequality follows from the fact that $u$ and $v$ are minimizers for the problem \eqref{minW11}.
 This proves \eqref{linA}, whence
\begin{align*}
    \int_{\Omega}|\tfrac12\nabla u+ \tfrac12\nabla v+\xva|\ud\LL^{2n}=\frac12\int_{\Omega}|\nabla u+     \xva|\ud\LL^{2n}+\frac12\int_{\Omega}|\nabla v+\xva|\ud\LL^{2n}\,.
\end{align*}
This in turn implies that
\[
    |(\nabla u+\xva)+ (\nabla v+\xva)|=|\nabla u+\xva|+|\nabla v+\xva|\quad \text{a.e. in }\Omega,
\]
i.e., $\nabla u+\xva$ and $\nabla v+\xva$ are  parallel (and with the same direction) $\LL^{2n}$-a.e. in $\Omega$. In particular, by Lemma \ref{starop} (i)
we obtain
\begin{align*}
0=(\nabla u+\xva)^\ast\cdot(\nabla v+\xva) =    ((\nabla u)^*-\xvett)\cdot(\nabla v+\xva)\quad \mbox{a.e. in}\ \Omega\,.
\end{align*}
Thus
\begin{align}\label{intg1}
\int_{\Omega_1}((\nabla u)^*-\xvett)\cdot(\nabla v+\xva)\ud\LL^{2n}=0,
\end{align}
where $\Omega_1:=\{z\in\Omega: u(z)>v(z)\}$. Let us expand \eqref{intg1} to get
\begin{align}\label{intg2}
\int_{\Omega_1}(\nabla u)^*\cdot \nabla v\ \ud\LL^{2n}+\int_{\Omega_1}(\nabla u)^*\cdot \xva\ud\LL^{2n}-\int_{\Omega_1}\nabla v\cdot\xvett\ \ud\LL^{2n}-\int_{\Omega_1}\xvett\cdot\xva \ud\LL^{2n}=0\,.
\end{align}
Lemma \ref{starop} gives
\begin{align*}
(\nabla u)^*\cdot \xva=\nabla u\cdot \xvett \quad  \mbox{and}\quad   \xvett\cdot \xva=0,
\end{align*}
so that \eqref{intg2} becomes
\begin{align}\label{intg3}
-\int_{\Omega_1}(\nabla u)^*\cdot (\nabla u-\nabla v)\ud\LL^{2n}+\int_{\Omega_1}\xvett\cdot (\nabla u-\nabla v)\ud\LL^{2n}=0
\end{align}
where we also used the fact that $(\nabla u)^\ast\cdot \nabla u=0$. By the classical Stampacchia Theorem, we have $(u-v)^+\in W^{1,p}(\Omega)$ and
$$
\nabla (u-v)^+=(\nabla (u-v))\,\chi_{\Omega_1} \quad \mbox{a.e. in }\Omega,
$$
hence \eqref{intg3} can be written as
\[
-\int_{\Omega}(\nabla u)^*\cdot \nabla (u-v)^+ \ud\LL^{2n}+\int_{\Omega}\xvett\cdot \nabla (u-v)^+\ud\LL^{2n}=0.
\]
Integrating by parts and using the fact that $(u-v)^+_{|\partial\Omega}=0$ because $u_{|\partial\Omega}=v_{|\partial\Omega}=\varphi$, we obtain
\begin{equation}\label{eq}
\begin{split}
0=&-\int_{\Omega}(\nabla u)^*\cdot \nabla (u-v)^+ \ud\LL^{2n}+\int_{\Omega}\xvett\cdot \nabla (u-v)^+\ud\LL^{2n}\\
=&-\int_{\Omega}(\nabla u)^*\cdot \nabla (u-v)^+ \ud\LL^{2n}-\int_{\Omega}(u-v)^+\,\dive\xvett\ud\LL^{2n}\\
=&-\int_{\Omega}(\nabla u)^*\cdot \nabla (u-v)^+ \ud\LL^{2n}-2n\int_{\Omega} (u-v)^+\ud\LL^{2n}
\end{split}
\end{equation}
We claim that
\begin{align}\label{eqfin}
\int_{\Omega}(\nabla u)^*\cdot \nabla (u-v)^+ \ud\LL^{2n}=0.
\end{align}
To this end, consider a sequence $(u_k)_{k\in\bbN}$ such that
\[
\text{$u_k\in C^{\infty}(\Omega)\cap W^{1,p'}(\Omega)$\quad and\quad $\nabla u_k \stackrel{*}{\rightharpoonup} \nabla u$ in $L^{p'}(\Omega)$ as $k\to +\infty$.}
\]
We have also ${(\nabla u_k)}^\ast \stackrel{*}{\rightharpoonup} \nabla u^\ast$, thus
\begin{align*}
    \int_{\Omega}(\nabla u)^*\cdot \nabla (u-v)^+ \ud\LL^{2n}&=\lim_{k\to\infty}\int_{\Omega}{(\nabla u_k)}^*\cdot \nabla (u-v)^+ \ud\LL^{2n}\\
    &=\int_{\Omega}\dive((\nabla u)^*)\:(u-v)^+ \ud\LL^{2n}\\
    &=0
\end{align*}
by Lemma \ref{starop} (iii). By \eqref{eq} and \eqref{eqfin} we deduce that $(u-v)^+=0$ a.e. on $\Omega$.

On considering $\Omega_2:=\{z\in\Omega: v(z)>u(z)\}$ in place of $\Omega_1$, one can similarly prove that $(u-v)^-=0$ a.e. on $\Omega$. This completes the proof.
\end{proof}

We introduce now some notations that will be useful also in the proof of the main theorem of the paper.
Given a subset $\Omega\subset\bbR^{2n}$, a function $u:\Omega\to\R$, a vector $\tau\in\R^{2n}$ and $\xi\in\R$ we set
\[
\begin{split}
& \Omega_{\tau}:=\{z\in\bbR^{2n}: z+\tau\in \Omega\}\\
& u_{\tau}(z):=u(z+\tau),\quad z\in \Omega_\tau\\
& \atx{u}(z):=u_{\tau}(z)+2\left\langle \tau^*, z\right\rangle+\xi,\quad z\in \Omega_\tau\,.
\end{split}
\]
It is easily seen that, given $\Omega$ open and $u\in \BV(\Omega)$, then both $u_\tau$ and $\atx u$ belong to $\BV(\Omega_\tau)$. Moreover, if $\Omega$ is bounded with Lipschitz regular boundary one has also
\begin{equation}\label{trennno}
(\atx u)_{|\partial(\Omega_\tau)}= (u_{|\partial\Omega})_{\tau}+2\left\langle \tau^*,\cdot\right\rangle+\xi=\atx{(u_{|\partial\Omega})}\,.
\end{equation}

\begin{oss}{\rm
The family of functions $\atx{u}$ has a precise meaning from the viewpoint of Heisenberg groups geometry. Indeed, it is a matter of computations to observe that the $t$-subgraph $E^t_{\atx u}$ of $\atx{u}$ coincides with the left translation $(-\tau,\xi)\cdot E^t_u$ (according to the group law) of the $t$-subgraph $E^t_u$ of $u$ by the element $(-\tau,\xi)\in \bbH^n$.
}\end{oss}

\begin{lem}\label{lemutile}
Let $\Omega\subset\R^{2n}$ be a bounded open set with Lipschitz regular boundary, $\varphi\in L^1(\partial\Omega)$, $\tau\in\R^{2n}$ and $\xi\in\R$. Then
\[
\area_{\atx{\varphi},\Omega_\tau}(\atx u)=\area_{\varphi,\Omega}(u)\quad\forall\: u\in \BV(\Omega)\,.
\]
\end{lem}
\begin{proof}
Using e.g. \cite[Remark 3.18]{ambfuspal}, it is not difficult to prove that $D u_\tau=\ell_{\tau\#}(Du)$, where $\ell_\tau$ is the translation $z\mapsto z-\tau$ and $\ell_{\tau\#}$ denotes the push-forward of measures via $\ell_\tau$. In particular
\[
\nabla u_\tau=(\nabla u)_\tau=\nabla u\circ\ell_\tau^{-1}\quad\text{and}\quad D^s u_\tau=\ell_{\tau\#}(D^su)\,,
\]
hence
\[
D\atx u = \big(\nabla u\circ\ell_\tau^{-1}+2\tau^\ast\big)\LL^{2n} + \ell_{\tau\#}(D^su)\,.
\]
Therefore
\[
\begin{split}
& \area_{\atx{\varphi},\Omega_\tau}(\atx u)\\
= & \int_{\Omega_\tau} |(\nabla u\circ\ell_\tau^{-1})+2\tau^\ast+\xva|\ud\LL^{2n} + |\ell_{\tau\#}(D^su)|(\Omega_\tau) + \int_{\partial\Omega_\tau}|(\atx u)_{|\partial(\Omega_\tau)}-\atx\varphi|\ud\HH^{2n-1}.
\end{split}
\]
We now use \eqref{trennno} and the equality
\[
2\tau^\ast+\xva(z) = 2(\tau+z)^\ast = (\xvett^\ast\circ \ell_\tau^{-1})(z)\quad\forall z\in\R^{2n}
\]
to get, with a change of variable,
\[
\begin{split}
& \area_{\atx{\varphi},\Omega_\tau}(\atx u)\\
= & \int_{\Omega_\tau} |\nabla u+\xva|\circ\ell_\tau^{-1}\ud\LL^{2n} + |\ell_{\tau\#}(D^su)|(\ell_\tau(\Omega)) + \int_{\partial\Omega_\tau}\big|\big(u_{|\partial\Omega}-\varphi\big)_\tau\big|\ud\HH^{2n-1}\\
=& \int_\Omega |\nabla u + \xva|\ud\LL^{2n} + |D^su|(\Omega) + \int_{\partial\Omega}|u_{|\partial\Omega}-\varphi|\ud\HH^{2n-1}\\
= & \area_{\varphi,\Omega}(u)\,.
\end{split}
\]
\end{proof}

\begin{cor}\label{tildemax}
Under the same assumptions of Lemma \ref{lemutile}: if $\overline{u}$ and $\underline{u}$ are as in Proposition \ref{esistmax}, then $\atx{(\overline{u})},\atx{(\underline{u})}\in \cc_{\atx{\varphi}}$ and
\[
\atx{(\overline{u})} \leq u\leq \atx{(\underline{u})}\quad \LL^{2n}\mbox{-a.e. in }\Omega_{\tau}, \forall u\in \cc_{\atx{\varphi}}\,.
\]
\end{cor}

The next theorem states that, whenever we fix an affine boundary
datum, the functional with the penalization on the boundary has a
unique minimizer that is the affine function itself. This is one of the main
tools in the proof of Theorem \ref{mainteo-ex6.7}.

\begin{teo}\label{affineMIN} Let $\Omega\subset\bbR^{2n}$
be a bounded open set with Lipschitz regular boundary and let
$L:\bbR^{2n}\to\bbR$ be an affine function, i.e.,
$L(z)=\langle a,z\rangle+b$ for some $a\in\R^{2n},b\in\bbR$. Then $L$
is the unique solution of
\begin{align}\label{min}
    \min \{\area_{L,\Omega}(u): u\in \BV(\Omega)\}.
\end{align}
\end{teo}
\begin{proof} We divide the proof into several steps.

{\em Step 1: }reduction to the case $L=0$.

\noindent Setting $\tau:=a^*/2\in\R^{2n}$ and $\xi=-b$, one has $\atx L\equiv 0$. By Lemma \ref{lemutile} and Corollary \ref{tildemax}, the fact that $L$ is the unique solution of \eqref{min} is equivalent to the fact that 0 is the unique minimizer of the problem
\[
\min \{\area_{0,\Omega_\tau}(u): u\in \BV(\Omega_\tau)\}.
\]
In view of this, we can henceforth assume that $L=0$.

{\em Step 2: }$L=0$ is a minimizer for \eqref{min}.

\noindent Let $u\in \BV(\Omega)$; by the dominated convergence theorem we have
\[
\begin{split}
\int_{\Omega} \left\langle \sigma_u,\tfrac{\xva}{|\xva|}\right\rangle\ud |Du| = & \lim_{\epsilon\to 0} \int_{\Omega} \left\langle \sigma_u,\tfrac{\xva}{|\xva|+\epsilon}\right\rangle\ud |Du|\\
= & \lim_{\epsilon\to 0}\left[ -\int_\Omega u\:\dive\!\left(\tfrac{\xva}{|\xva|+\epsilon}\right)\ud\LL^{2n} + \int_{\partial\Omega}u_{|\partial\Omega}\left\langle \nu_\Omega,\tfrac{\xva}{|\xva|+\epsilon}\right\rangle\ud\HH^{2n-1}\right] \\
= \:& \int_{\partial\Omega} u_{|\partial\Omega} \left\langle \nu_\Omega,\tfrac{\xva}{|\xva|}\right\rangle \ud\HH^{2n-1} \,,
\end{split}
\]
where we have used the fact that $\dive\!\big(\tfrac{\xva}{|\xva|+\epsilon}\big)=0$. Thus
\begin{equation}\label{eq:alina}
\begin{split}
\area_{0,\Omega}(u) & =  \int_{\Omega} |\nabla u+\xva|\ud\LL^{2n}+|D^s u|(\Omega)+\int_{\partial\Omega}|u_{|\partial\Omega}|\ud\HH^{2n-1}\\
    &\geq \int_{\Omega} \left\langle \nabla u+\xva,\tfrac{\xva}{|\xva|}\right\rangle\ud \LL^{2n}+ \int_{\Omega}\left\langle \sigma_u,\tfrac{\xva}{|\xva|} \right\rangle\ud |D^su| + \int_{\partial\Omega}|u_{|\partial\Omega}|\ud\HH^{2n-1}\\
    &=\int_{\Omega} \left\langle \sigma_u,\tfrac{\xva}{|\xva|}\right\rangle\ud |Du|+ \int_{\Omega}\left\langle \xva,\tfrac{\xva}{|\xva|} \right\rangle\ud \LL^{2n} +\int_{\partial\Omega}|u_{|\partial\Omega}|\ud\HH^{2n-1}\\
& = \int_{\Omega}|\xva|\LL^{2n} +\int_{\partial\Omega}|u_{|\partial\Omega}|\left( 1+\text{sgn}(u_{|\partial\Omega})\left\langle \nu_\Omega,\tfrac{\xva}{|\xva|}\right\rangle\right) \ud\HH^{2n-1}\\
& \geq \int_{\Omega}|\xva|\LL^{2n}= \area_{0,\Omega}(0)\,,
\end{split}
\end{equation}
which proves that $L=0$ is a minimizer of \eqref{min}.

{\em Step 3: }if $\Omega=B(0,R)$ for some $R>0$, then $L=0$ is the unique minimizer of \eqref{min}.

\noindent Since $\Omega$ is a ball centered at the origin, we have $\langle \nu_\Omega,\tfrac{\xva}{|\xva|}\rangle=0$ and, discarding its third line, \eqref{eq:alina} can be rewritten as
\begin{equation}\label{eq:alina22}
\begin{split}
\area_{0,\Omega}(u) & =  \int_{\Omega} |\nabla u+\xva|\ud\LL^{2n}+|D^s u|(\Omega)+\int_{\partial\Omega}|u_{|\partial\Omega}|\ud\HH^{2n-1}\\
&\geq \int_{\Omega} \left\langle \nabla u+\xva,\tfrac{\xva}{|\xva|}\right\rangle\ud \LL^{2n}+ \int_{\Omega}\left\langle \sigma_u,\tfrac{\xva}{|\xva|} \right\rangle\ud |D^su| + \int_{\partial\Omega}|u_{|\partial\Omega}|\ud\HH^{2n-1}\\
& = \int_{\Omega}|\xva|\LL^{2n} +\int_{\partial\Omega}|u_{|\partial\Omega}|\ud\HH^{2n-1}\\
& \geq \int_{\Omega}|\xva|\LL^{2n}= \area_{0,\Omega}(0)\,,
\end{split}
\end{equation}
Let $u\in \BV(\Omega)$ be a minimizer for \eqref{min}; then (by Step 2) $\area_{0,\Omega}(u)=\area_{0,\Omega}(0)$  and the two inequalities in \eqref{eq:alina22} must be equalities. In particular, one has $u_{|\partial\Omega}=0$ and
\[
\begin{split}
& \int_{\Omega} |\nabla u+\xva|\ud\LL^{2n} =  \int_{\Omega} \left\langle \nabla u+\xva,\tfrac{\xva}{|\xva|}\right\rangle\ud \LL^{2n}\,,\\
&|D^s u|(\Omega) =  \int_{\Omega}\left\langle \sigma_u,\sigma_u \right\rangle\ud |D^su| = \int_{\Omega}\left\langle \sigma_u,\tfrac{\xva}{|\xva|} \right\rangle\ud |D^su|\,,
\end{split}
\]
so that $\sigma_u=\tfrac{\xva}{|\xva|}$ $|D^su|$-a.e. and there exists a measurable function $\lambda:\Omega\to[0,+\infty)$ such that
\[
\sigma_u|\nabla u|+\xva=\nabla u+\xva=\lambda\tfrac{\xva}{|\xva|}\quad\LL^{2n}\text{-a.e on }\Omega\,.
\]
All in all, there exists a $|Du|$-measurable function $\tilde\lambda:\Omega\to\R$ such that
\begin{align}\label{ugu3}
    \sigma_u=\tilde\lambda\xva \quad |Du|\mbox{-a.e. in}\ \Omega.
\end{align}
We claim that, up to the choice of a representative, any function $u$ satisfying \eqref{ugu3} is 0-homogeneous, i.e., it satisfies $u(z)=u(tz)$ for any $z\in B(0,R)$ and $t\in(0,1)$; roughly speaking, \eqref{ugu3} says indeed that $u$ has null radial derivative, which suggests its 0-ho\-mo\-ge\-nei\-ty. Recalling that $u_{|\partial\Omega}=0$, this would be enough to conclude that $u\equiv0$.

Consider the map
\[
\begin{split}
F:\, & [0,R)\times \bbS^{2n-1}\to \Omega=B(0,R)\\
& (\rho,\theta)\longmapsto \rho\theta
\end{split}
\]
The claimed 0-homogeneity of $u$ is clearly equivalent to the fact that $ u_0:=u\circ F:[0,R]\times \bbS^{2n-1}\to\R$ admits a representative which does not depend on $\rho$. Thus, it will be enough to prove that for any $ f_0\in C^{\infty}_c ((0,R)\times \bbS^{2n-1})$ there holds
\begin{equation}\label{antipastopiemontese}
\int_{(0,R)\times \bbS^{2n-1}}  u_0\: \frac{\partial  f_0}{\partial\rho}\,\ud(\LL^1\otimes\mu) = 0\,,
\end{equation}
where $\mu$ is the Haar measure on $\bbS^{2n-1}$. Define $f\in C^\infty_c(B(0,R)\setminus\{0\})$ by $f(z):=( f_0 \circ F^{-1})(z) = f_0(|z|,\tfrac{z}{|z|})$ for any $z\in B(0,R)\setminus\{0\}$; notice that
\[
\bigg(\Big( \frac{\partial  f_0}{\partial\rho}\Big)\circ F^{-1}\bigg)(z) = \frac{\partial f_0}{\partial\rho}\Big( |z|,\frac{z}{|z|}\Big) = \Big\langle \nabla f(z),\frac{z}{|z|}\Big\rangle\,.
\]
By a change of variable we get
\[
\begin{split}
& \int_{(0,R)\times \bbS^{2n-1}}  u_0\: \frac{\partial f_0}{\partial\rho}\ud(\LL^1\otimes\mu)\\
 = & \int_{(0,R)\times \bbS^{2n-1}}  u( F(\rho,\theta))\:\bigg(\Big( \frac{\partial  f_0}{\partial\rho}\Big) \circ F^{-1}\bigg)(F(\rho,\theta)) \:\frac{1}{\rho^{2n-1}} \,\rho^{2n-1}\,\ud(\LL^1\otimes\mu)(\rho,\theta)\\
 = & \int_{B(0,R)} u(z)\,\big\langle \nabla f(z),\frac{z}{|z|}\big\rangle\,\frac{1}{|z|^{2n-1}}\ud\LL^{2n}(z)\\
 = & \sum_{i=1}^{2n} \int_{B(0,R)} u(z)\,\frac{z_i}{|z|^{2n}}\,\frac{\partial f}{\partial z_i}(z)\ud\LL^{2n}(z)\\
 = & -\int_{B(0,R)} u(z)\,f(z)  \left[\sum_{i=1}^{2n} \frac{\partial}{\partial z_i}\left(\frac{z_i}{|z|^{2n}}\right)\right]\ud\LL^{2n}(z)
- \int_{B(0,R)} f(z)\langle \sigma_u(z),z\rangle \frac{1}{|z|^{2n}}\ud|Du|\,.
\end{split}
\]
Our claim \eqref{antipastopiemontese} is then a consequence of the equality
\[
\sum_{i=1}^{2n} \frac{\partial}{\partial z_i}\left(\frac{z_i}{|z|^{2n}}\right) = 0
\]
and the fact that $\sigma_u(z)=(\tilde\lambda \xva)(z)$ is $|Du|$-a.e. orthogonal to $z$.

{\em Step 4: }$L=0$ is the unique minimizer of \eqref{min} for general $\Omega$.

\noindent Let $u\in \BV(\Omega)$ be a minimizer of \eqref{min} and let $R>0$ be such that $\Omega\Subset B(0,R)$. Let us define
\[u_0(z):=
\left\{\begin{array}{ll}
u(z) & \text{if }z\in\Omega\\
0 & \text{if }z\in B(0,R)\setminus\overline\Omega\,.
\end{array}\right.
\]
By Step 2, also $L=0$ is a minimizer, i.e., $\area_{0,\Omega}(u)=\area_{\Omega}(0)$; thus
\begin{align*}
    \area_{0,B(0,R)}(u_0)& =\int_{\Omega}|\nabla u+\xva|\ud\LL^{2n}+|D^su|(\Omega)+\int_{B(0,R)\setminus\overline\Omega}|\xva|\ud\LL^2+ |D^su_0|(\partial\Omega)\\
     &=\area_{\Omega}(u)+\area_{B(0,R)\setminus\overline\Omega}(0)+\int_{\partial\Omega}|u_{|\partial\Omega}|\ud\HH^{2n-1}\\
    &= \area_{0,\Omega}(u) +  \area_{B(0,R)\setminus\overline\Omega}(0)\\
    &= \area_{\Omega}(0) +  \area_{B(0,R)\setminus\overline\Omega}(0) = \area_{B(0,R)} (0)=\area_{0,B(0,R)}(0)\,.
\end{align*}
Therefore, $u_0$ is a minimizer of $\area_{0,B(0,R)}$; by Step 3, this implies that $u_0=0$, i.e., that $u=0$ $\LL^{2n}$-a.e. on $\Omega$, as desired.
\end{proof}

The next Corollary is a special comparison principle for affine function and, in particular shows that affine functions satisfy a comparison principle both from above and from below.

\begin{cor}\label{cpaffine}
Let $\Omega\subset\bbR^{2n}$ be a bounded open set with Lipschitz boundary, $\varphi\in L^1(\partial\Omega)$ and $L:\bbR^{2n}\to\bbR$ be an affine function, i.e., $L(z)=\langle a,z\rangle+b$ for some $a\in\R^{2n},b\in\bbR$.
\begin{enumerate}[i)]
\item Assume that $\varphi\leq L$ $\HH^{2n-1}$-a.e. on $\partial\Omega$. Then, for any minimizer $u\in \cc_{\varphi}$ of $\area_\varphi$, we have $u\le L$ $\LL^{2n}$-a.e. in $\Omega$.
\item Assume that that $\varphi\geq L$ $\HH^{2n-1}$-a.e. on $\partial\Omega$. Then, for any minimizer $u\in \cc_{\varphi}$ of $\area_\varphi$, we have $u\ge L$ $\LL^{2n}$-a.e. in $\Omega$.
\end{enumerate}
\end{cor}
\begin{proof}
Both claims follow immediately from Theorem \ref{confro} when we
observe that the set $\cc_L$ consists of just one element
that is $L$ itself, so that, following
the notations of Proposition \ref{esistmax}, $L=\overline L=\underline L$.
\end{proof}

\section{The Bounded Slope Condition}\label{ultSec}
We recall the well-known definition of Bounded Slope Condition (see \cite{HartStamp}) for boundary data. In particular we refer to
\cite{Giusti2} also for some classical results that we will summarize next.

\begin{defi}\label{BSC}
{\rm We say that a function $\varphi:\partial\Omega\to\bbR$ satisfies the bounded slope condition with constant $Q>0$ ($Q$-B.S.C.
for short, or simply B.S.C. when the constant $Q$ does not play any role) if for every $z_0\in\partial\Omega$ there exist two affine functions $w^+_{z_0}$ and $w^-_{z_0}$ such that
\begin{equation}\label{stimali}
\begin{split}
&w^-_{z_0}(z)\leq \varphi(z)\leq w^+_{z_0}(z) \quad \forall z\in\partial\Omega,\\
&w^-_{z_0}(z_0)=\varphi(z_0)=w^+_{z_0}(z_0)\\
&\lip(w^-_{z_0})\leq Q\quad \mbox{and}\quad
\lip(w^+_{z_0})\leq Q,
\end{split}
\end{equation}
where $\lip(w)$ denotes the Lipschitz constant of $w$.
}\end{defi}

We also recall that a set $\Omega\subset \bbR^{2n}$ is said to be uniformly convex if there exists a positive constant $C=C(\Omega)$
and, for each $z_0\in\partial\Omega$, an hyperplane $\Pi_{z_0}$
passing through $z_0$ such that $$|z-z_0|^2\leq C\,  dist( z,
\Pi_{z_0})\quad \forall z\in\partial\Omega,$$ where $dist(z, \Omega):=\inf\{|z-w|\ |\ w\in
\Omega\}$. It is worth noticing
that, if $\partial\Omega$ is of class $C^2$, this condition holds if
and only if all principal curvatures of $\partial\Omega$ are
strictly positive, see \cite{Giusti2} for details.

\begin{oss}\label{propBSC}
{\rm We collect here some facts on the B.S.C.
\begin{enumerate}
\item[a)] If $\varphi:\partial\Omega\to\bbR$ satisfies the
B.S.C. and is not affine, then $\Omega$ has to be convex (see \cite{Giusti2}) and $\varphi$ is Lipschitz continuous on $\partial\Omega$. Moreover, if
$\partial\Omega$ has flat faces, then $\varphi$  has to be affine on them.
\end{enumerate}
This property seems to say that the B.S.C. is a quite restrictive
assumption. Anyhow the following one, due to M. Miranda \cite{Miranda} (see also \cite[Theorem 1.1]{Giusti2}), shows
that the class of functions satisfying the B.S.C. on a uniformly
convex set is quite large.
\begin{enumerate}
\item[b)] Let $\Omega\subset\bbR^n$ be open, bounded and uniformly convex; then
 every $\varphi\in C^{1,1}(\bbR^n)$
satisfies the B.S.C. on $\partial\Omega$.
\end{enumerate}
}\end{oss}

We denote by $f, g$ the functions defined, respectively, by
$f(z):=\sup_{z_0\in\partial\Omega} w^-_{z_0}(z)$ and
$g(z):=\inf_{z_0\in\partial\Omega} w^+_{z_0}(z)$. We underline that $f$
is a convex function, $g$ is a concave function and both are
Lipschitz with Lipschitz constant not greater than $Q$.
\begin{lem}\label{datobordo}
Let $\Omega\subset\bbR^{2n}$ be an open bounded set with Lipschitz regular boundary; assume that $\varphi\in L^1(\partial\Omega)$ satisfies the $Q$-B.S.C. Then, for any $u\in \cc_{\varphi}$ there holds
\begin{enumerate}[i)]
\item   $u_{|\partial\Omega}=\varphi$;
\item $f\le u\le g$ $\LL^{2n}$-a.e. in $\Omega$;
\item $u$ is also a minimizer of  $\area_{\Omega}$  in $\BV(\Omega)$ with $u_{|\partial\Omega}=\varphi$.
\end{enumerate}
\end{lem}
\begin{proof}
$i)$ For every $z_0\in\partial\Omega$, let $w_{z_0}^+$ and
$w_{z_0}^{-}$ be as in Definition \ref{BSC}. By Corollary
\ref{cpaffine}, we have that $w_{z_0}^-\le u\le w_{z_0}^+$
$\LL^{2n}$-a.e. in $\Omega$. Recalling \eqref{stimali} we obtain
\begin{align}\label{gho2}
 |u(z)-\varphi(z_0)|\leq Q |z-z_0|\quad \LL^{2n}\mbox{-a.e.}\ z\in\Omega, \forall z_0\in\partial\Omega.
\end{align}
Therefore,
\begin{align}\label{mediatr}
\frac{1}{\rho^{2n}}\int_{\Omega\cap B(z_0,\rho)}|u-\varphi(z_0)|\ud\LL^{2n}\leq \frac{Q}{\rho^{2n}}\int_{\Omega\cap B(z_0,\rho)}|z-z_0|\ud\LL^{2n}(z)\leq Q\rho
\end{align}
and letting $\rho\to 0^+$ in (\ref{mediatr}) we conclude that
$u_{|\partial\Omega}=\varphi$.

$ii)$ Fix  a Lebsgue point $\bar z\in\Omega$ of $u$.
Since $f$ is a convex function, there exists $\xi\in\bbR^{2n}$ such that $f(z)\ge
f(\bar z)+\xi\cdot(z-\bar z):=h(z)$ for every $z\in\overline\Omega$. The
function $h$ is affine and $h\le\varphi$ on $\partial\Omega$; then
Corollary \ref{cpaffine} implies that $u\ge h$ $\LL^{2n}$-a.e. in
$\Omega$. Considering the mean integral on a ball centered at $\bar z$
we obtain
\[
\fint_{B(\bar z,\rho)} u(z) d\LL^{2n}\ge  \fint_{B(\bar z,\rho)} h(z) d\LL^{2n}
\]
and, passing to the limit as $\rho\to0^+$, we get $u(\bar z)\ge
f(\bar z)$. One can argue in a similar way to prove that $u\le g$
$\LL^{2n}$-a.e. in $\Omega$.

Finally, the proof of $iii)$ is straightforward.
\end{proof}

\begin{oss}\label{sottodomini}{\rm
If $\Omega'\subset\Omega$ are open bounded domains with Lipschitz regular boundary and $u\in \BV(\Omega)$, we use the notation $\area_{u,\Omega'}$ to denote the functional $\area_{u_{|\partial\Omega'},\Omega'}$. Let us prove that, if $u$ is a minimizer of $\area_{\varphi,\Omega}$ with $\varphi=u_{|\partial\Omega}$, then $u$ is also a minimizer of $\area_{u,\Omega'}$.

Let us  write $\Gamma:=\partial\Omega'\cap\Omega$ and $\partial\Omega=\Delta_1\cup\Delta_2$, where
\[
\Delta_1:= \partial\Omega\cap\partial\Omega'\quad\text{and}\quad \Delta_2:= \partial\Omega\setminus\partial\Omega'\,.
\]
Notice that $\partial\Omega'=\Gamma\cup\Delta_1$. We also denote by $u_i,u_o:\Gamma\to\R$ the ``inner'' and ``outer'' (with respect to $\Omega'$) traces of $u$ on $\Gamma$, i.e.,
\[
u_i:=(u_{|\partial\Omega'})\res\Gamma\quad\text{and}\quad u_o:=(u_{|\partial(\Omega\setminus\overline{\Omega'})})\res\Gamma\,.
\]
Assume by contradiction that $u$ is not a minimizer of $\area_{u,\Omega'}$; then, there exists $v\in \BV(\Omega')$ such that
\begin{equation}\label{ristorantebrasiliano}
\begin{split}
0 & < \area_{u,\Omega'}(u)-\area_{u,\Omega'}(v)\\
& = \area_{\Omega'}(u) - \area_{\Omega'}(v) - \int_{\partial\Omega'}|v_{|\partial\Omega'}-u_{|\partial\Omega'}|\ud\HH^{2n-1}\\
& = \area_{\Omega'}(u) - \area_{\Omega'}(v) - \int_{\Gamma}|v_{|\partial\Omega'}-u_i|\ud\HH^{2n-1}- \int_{\Delta_1}|v_{|\partial\Omega'}-\varphi|\ud\HH^{2n-1}\,.
\end{split}
\end{equation}
We will reach a  contradiction if we show that the function $w\in\BV(\Omega)$ defined by
\[
w:=v\text{ on }\Omega',\quad w:=u\text{ on }\Omega\setminus\Omega'
\]
satisfies $\area_{\varphi,\Omega}(u)- \area_{\varphi,\Omega}(w)>0$.

Let us compute
\[
\begin{split}
\area_{\varphi,\Omega}(u) & = \area_{\Omega}(u) =  \area_{\Omega'}(u) + \area_{\Omega\setminus\overline{\Omega'}}(u) + |D^s u|(\Gamma)\\
& =   \area_{\Omega'}(u) + \area_{\Omega\setminus\overline{\Omega'}}(u) + \int_\Gamma |u_o-u_i|\ud\HH^{2n-1}
\end{split}
\]
(we have used the assumption $\varphi=u_{|\partial\Omega}$) and
\[
\begin{split}
\area_{\varphi,\Omega}(v) &=   \area_{\Omega'}(v) + \area_{\Omega\setminus\overline{\Omega'}}(u) + |D^s w|(\Gamma)+ \int_{\partial\Omega} |w_{|\partial\Omega}-\varphi|\ud\HH^{2n-1}\\
&=   \area_{\Omega'}(v) + \area_{\Omega\setminus\overline{\Omega'}}(u) + \int_\Gamma |v_{|\partial\Omega'}-u_o|\ud\HH^{2n-1} + \int_{\Delta_1} |v_{|\partial\Omega}-\varphi|\ud\HH^{2n-1}\,.
\end{split}
\]
Therefore
\[
\begin{split}
& \area_{\varphi,\Omega}(u) -\area_{\varphi,\Omega}(v)\\
= \,& \area_{\Omega'}(u)-  \area_{\Omega'}(v) + \int_\Gamma \Big(|u_o-u_i| - |v_{|\partial\Omega'}-u_o| \Big)\ud\HH^{2n-1} - \int_{\Delta_1} |v_{|\partial\Omega}-\varphi|\ud\HH^{2n-1}\\
\geq\, & \area_{\Omega'}(u) - \area_{\Omega'}(v) - \int_{\Gamma}|v_{|\partial\Omega'}-u_i|\ud\HH^{2n-1}- \int_{\Delta_1}|v_{|\partial\Omega'}-\varphi|\ud\HH^{2n-1}\\
>\, & 0
\end{split}
\]
by \eqref{ristorantebrasiliano}, as desired.
}\end{oss}

We are now in position to prove our main result.

\begin{proof}[Proof of Theorem \ref{mainteo-ex6.7}] We divide the proof into several steps.

{\em Step 1.} We denote by $\overline u$ the (pointwise a.e.) maximum of the minimizers of $\area_{\varphi,\Omega}$ in $\BV$ (see Theorem \ref{maxmin}). Lemma \ref{datobordo} implies that $f\le \overline u\le g$ $\LL^{2n}$-a.e. in $\Omega$ and $\overline u=\varphi=f=g$ on $\partial\Omega$; in particular, $\overline u$ is also a minimizer for \eqref{problema}.

Let $\tau\in\R^{2n}$ be such that $\Omega\cap\Omega_\tau\neq\emptyset$; following the notations introduced in Section \ref{uniqSec}, we consider the function $\overline u_{\tau,0}^*$, which we denote by $\overline u_\tau^*$ to simplify the notation. Let us consider the set $\Omega\cap\Omega_\tau$. By Remark \ref{sottodomini}, $\overline u$ is a minimizer of $\area_{\overline u,\Omega\cap\Omega_\tau}$ and, by Corollary \ref{tildemax} and Remark \ref{sottodomini}, $\ut$ is a minimizer of $\area_{\ut,\Omega\cap\Omega_\tau}$. Let $z\in\partial(\Omega\cap\Omega_\tau)$, then either $z\in\partial\Omega$ or $z\in\partial\Omega_\tau$.

\noindent If $z\in \partial \Omega$, then $z+\tau\in\overline\Omega$ and the inequality \eqref{gho2} in Lemma \ref{datobordo} implies that
\begin{equation}\label{disbordo}
\overline u(z)-Q|\tau|\le \overline u(z+\tau)\le \overline u(z)+Q|\tau|\,.
\end{equation}
Otherwise, $z\in\partial\Omega_\tau$ and $z=(z+\tau)-\tau\in\overline\Omega$, and Lemma \ref{datobordo} implies again \eqref{disbordo}.

So we have proved that \eqref{disbordo} holds for any $z\in\partial (\Omega\cap\Omega_\tau)$, hence
\[
\overline u(z)-Q|\tau|+2\langle\tau^*,z\rangle\le \overline u(z+\tau)+2\langle\tau^*,z\rangle\le \overline u(z)+Q|\tau|+2\langle\tau^*,z\rangle\,.
\]
Setting $M:=Q+2\sup_{z\in\Omega}|z|$, one has
\[
\overline u(z)-M|\tau|\le  \ut(z)\le \overline u(z)+M|\tau|\quad \text{for any }z\in\partial(\Omega\cap\Omega_\tau)
\]
and, by Corollary \ref{valbordo},
\[
\overline u(z)-M|\tau|\le  \ut(z)\le \overline u(z)+M|\tau|\quad \text{for $\LL^{2n}$-a.e. }z\in\Omega\cap\Omega_\tau\,.
\]
This is equivalent to
\[
\overline u(z)-M|\tau|-2\langle \tau^*,z\rangle\le  \overline u(z+\tau)\le \overline u(z)+M|\tau|-2\langle \tau^*,z\rangle\quad \text{for $\LL^{2n}$-a.e. }z\in\Omega\cap\Omega_\tau
\]
and, setting $K:=M+2\sup_{z\in\Omega}|z|$,
\[
\overline u(z)-K|\tau|\le  \overline u(z+\tau)\le \overline u(z)+K|\tau|\quad \text{for $\LL^{2n}$-a.e. }z\in\Omega\cap\Omega_\tau
\]

{\em Step 2.} We claim that the inequality $|\overline u(z)-\overline u(\bar z)|\le K|z-\bar z|$ holds for any Lebesgue points $z,\bar z$ of $\overline u$. We define $\tau:=\bar z-z$; then $\Omega\cap\Omega_\tau\neq\emptyset$ and, arguing as in Step 1, we obtain
\[
|\overline u(z'+\tau)-\overline u(z')|\le K|\tau| \quad\text{for $\LL^{2n}$-a.e. }z'\in\Omega\cap\Omega_\tau.
\]
Let $\rho>0$ be such that $B(z,\rho)\subset\Omega\cap\Omega_\tau$ and $B(\bar z,\rho)\subset\Omega\cap\Omega_\tau$; then
\[
\begin{split}
|\overline u(z)-\overline u(\bar z)|=&\left|\lim_{\rho\to 0}\left(\fint_{B(z,\rho)}\overline u(z') dz'-\fint_{B(\bar z,\rho)}\overline u(z') dz'\right)\right|\\
\le&\lim_{\rho\to 0}\fint_{B(z,\rho)}\left| \overline u(z') - \overline u(z'+\tau) \right|dz'\le K|z-\bar z|.
\end{split}
\]

{\em Step 3.} We have proved that $\overline u$, the  maximum of the minimizer of $\area_\varphi$, has a representative that is Lipschitz continuous on $\Omega$, with Lipschitz constant not greater than $K= Q+4\sup_{z\in\Omega}|z|$. The same argument leads to prove that $\underline u$, the minimum of the minimizers of $\area_\varphi$, has a representative that is Lipschitz continuous on $\Omega$, with Lipschitz constant not greater than $K$. The uniqueness criterion in Proposition \ref{unicitapp'} (with $p=1$) implies that $\overline u=\underline u$ $\LL^{2n}$-a.e. on $\Omega$. If $u$ is another minimizer of $\area_\varphi$, we have by Proposition \ref{esistmax} that $\underline u\le u\le\overline u$  $\LL^{2n}$-a.e. on $\Omega$. This concludes the proof.
\end{proof}

The following examples show that, at least in the case $n=1$, Theorem \ref{mainteo-ex6.7} is sharp, in the sense that minimizers might not be better than Lipschitz regular.

\begin{es}\label{es1}{\rm
It was proved\footnote{Up to an easy adaptation.} in \cite[Example 7.2]{CHY} that the Lipschitz function
\[
u(x,y):=\left\{\begin{array}{ll}
2xy & \text{if }y>0\\
0 & \text{if }y\leq 0
\end{array}\right.
\]
is a minimizer of $\area_{u_{|\partial\Omega},\Omega}$ on any bounded open set $\Omega$ with Lipschitz regular boundary. Let us prove that $u_{|\partial\Omega}$ satisfies the B.S.C. on the open set
\[
\Omega := \{(x,y)\in\R^2 : x^2-1 < y < 1-x^2 \}\,.
\]
Indeed, one can easily check that $u(x,y)=x(y-x^2+1)=:\varphi(x,y)$ for any $(x,y)\in\partial\Omega$; moreover, $\Omega$ is uniformly convex and $\varphi\in C^{\infty}(\R^2)$, thus $u_{|\partial\Omega}$ satisfies the B.S.C. on $\Omega$ because of Remark \ref{propBSC} (b). By Theorem \ref{mainteo-ex6.7}, $u$ is the unique minimizer of $\area_{\varphi,\Omega}$ on $\BV(\Omega)$; notice that $u$ is not better than Lipschitz continuous on $\Omega$.
}\end{es}

\begin{es}{\rm
The previous example provides a nonsmooth minimizer of $\area$ on a nonsmooth domain; it is anyway possible to exhibit nonsmooth minimizers also on smooth domain. Indeed, it was proved in \cite[Example 3.4]{R2} that the $C^{1,1}$ function $u(x,y):=-2xy + y|y|$ minimizes $\area$ (under boundary conditions given by $u$ itself) on any bounded domain $\Omega\subset\R^2$ with Lipschitz regular boundary. Notice that, by Remark \ref{propBSC} (b), $u$ satisfies a B.S.C. on any smooth and uniformly convex domain.
}\end{es}



\end{document}